\pdfoutput=1
\RequirePackage{ifpdf}
\ifpdf 
\documentclass[pdftex]{sigma}
\else
\documentclass{sigma}
\fi

\numberwithin{equation}{section}

\newtheorem{Theorem}{Theorem}[section]

\newtheorem{Lemma}[Theorem]{Lemma}
\newtheorem{Proposition}[Theorem]{Proposition}
{ \theoremstyle{definition}
\newtheorem{Definition}[Theorem]{Definition}
\newtheorem{Example}[Theorem]{Example}
\newtheorem{Remark}[Theorem]{Remark} }

\def\Weyl#1#2{{\mathcal{B}}(#1;#2)}
\def\AWeyl#1#2{{\mathcal{A}}(#1;#2)}
\def\rhom#1#2#3{{\operatorname{Hom}_{#1}(#2,#3)}}

\begin{document}

\newcommand{\arXivNumber}{1602.07456}

\renewcommand{\PaperNumber}{059}

\FirstPageHeading

\ShortArticleName{Noncommutative Dif\/ferential Geometry of Generalized Weyl Algebras}

\ArticleName{Noncommutative Dif\/ferential Geometry\\ of Generalized Weyl Algebras}

\Author{Tomasz BRZEZI\'NSKI~$^{\dag\ddag}$}

\AuthorNameForHeading{T.~Brzezi\'nski}

 \Address{$^\dag$~Department of Mathematics, Swansea University, Singleton Park, Swansea SA2 8PP, UK}
 \EmailD{\href{mailto:T.Brzezinski@swansea.ac.uk}{T.Brzezinski@swansea.ac.uk}}

 \Address{$^\ddag$~Department of Mathematics, University of Bia{\l}ystok,\\
 \hphantom{$^\ddag$}~K.~Cio{\l}kowskiego 1M, 15-245 Bia{\l}ystok, Poland}

\ArticleDates{Received February 29, 2016, in f\/inal form June 14, 2016; Published online June 23, 2016}

\Abstract{Elements of noncommutative dif\/ferential geometry of ${\mathbb Z}$-graded generalized Weyl algebras ${\mathcal A}(p;q)$ over the ring of polynomials in two variables and their zero-degree subalgebras ${\mathcal B}(p;q)$, which themselves are generalized Weyl algebras over the ring of polynomials in one variable, are discussed. In particular, three classes of skew derivations of ${\mathcal A}(p;q)$ are constructed, and three-dimensional f\/irst-order dif\/ferential calculi induced by these derivations are described. The associated integrals are computed and it is shown that the dimension of the integral space coincides with the order of the def\/ining polynomial $p(z)$. It is proven that the restriction of these f\/irst-order dif\/ferential calculi to the calculi on ${\mathcal B}(p;q)$ is isomorphic to the direct sum of degree 2 and degree $-2$ components of~${\mathcal A}(p;q)$. A~Dirac operator for~${\mathcal B}(p;q)$ is constructed from a (strong) connection with respect to this dif\/ferential calculus on the (free) spinor bimodule def\/ined as the direct sum of degree 1 and degree $-1$ components of~${\mathcal A}(p;q)$. The real structure of ${\rm KO}$-dimension two for this Dirac operator is also described.}


\Classification{16S38; 58B32; 58B34}

\section{Introduction}\label{sec.intro}
In this paper we continue the study initiated in \cite{Brz:gen} of some aspects of the noncommutative geo\-met\-ry of a class of degree-one {\em generalized Weyl algebras}~\cite{Bav:gweyl} or {\em rank-one hyperbolic algebras}~\cite{LunRos:Kas} over the polynomial ring in one and two variables. We denote these algebras respectively by~$\Weyl p q$ and~$\AWeyl p q$, where $q$ is a non-zero element of the ground f\/ield~${\mathbb K}$ (of characteristic 0) and $p$ is a~polynomial in one variable. $\Weyl p q$ can be interpreted as the coordinate algebra of a noncommutative surface, while, as shown in~\cite{Brz:gen}, $\AWeyl p q$ is a~noncommutative principal circle bundle over~$\Weyl p q$ provided zero is not a root of $p$. In algebraic terms this means that $\AWeyl p q$ is a~strongly ${\mathbb Z}$-graded algebra~\cite{Dad:gro} (or a Clif\/ford system~\cite{Dad:com}) and $\Weyl p q$ is isomorphic to the degree-zero part of $\AWeyl p q$. Perhaps best-known examples of such algebras are the standard quantum Podle\'s sphere \cite{Pod:sph} with the quantum ${\rm SU}_q(2)$-group~\cite{Wor:com} as a bundle over it (in this case $p$ is a~linear polynomial) which together form the quantum Hopf f\/ibration \cite{BrzMaj:gau}. More recent examples include quantum teardrops~\cite{BrzFai:tea} with quantum lens spaces~\cite{HonSzy:len} as bundles over them.

While in \cite{Brz:gen} we concentrated on establishing that $\AWeyl p q$ are principal comodule algebras~\cite{BrzHaj:Che} or noncommutative principal bundles~\cite{BrzMaj:gau} over $\Weyl p q$ and on some homological properties of~$\AWeyl p q$, the aim of the present article is to describe dif\/ferential aspects of these algebras. We begin by introducing three families of skew derivations on $\AWeyl p q$. Each family is labeled by a~scalar, and they include derivations of ${\mathbb Z}$-degrees~$-2$ (denoted by $\partial_+$), $0$ (denoted by~$\partial_0$) and~$2$ (denoted by $\partial_-$), respectively. The Leibniz rule for each of these skew derivations is twisted by an automorphism, denoted by~$\sigma_+$, $\sigma_0$ and $\sigma_-$, respectively. Choosing one member of each family, we form the system $(\partial_i,\sigma_i)$, $i=\pm, 0$ which is a~(diagonal) free twisted multi-derivation of size~3 and thus def\/ines a genuine derivation $d$ with values in a free left (and right) $\AWeyl p q$-module $\Omega$; see, e.g.,~\cite{BrzElK:int}. As a left $\AWeyl p q$-module, $\Omega$ is isomorphic to $\AWeyl p q^3$, while the right $\AWeyl p q$-action is twisted by automorphisms $\sigma_i$; see \eqref{right} for the def\/inition of the right action and \eqref{ext} for the def\/inition of $d$. Next we show in Theorem~\ref{thm.calc.A} that if the algebra $\AWeyl p q$ is {\em regular}, i.e., $p(0) \neq 0$ and $p(z)$ is coprime with its $q^2$-derivative (see Def\/inition~\ref{def.generic}) and $q^2\neq 1$, or $p$ has degree at most one and a non-zero constant term, then $\Omega = \AWeyl p q d\left(\AWeyl p q\right)$, that is $(\Omega, d)$ is a f\/irst-order dif\/ferential calculus over $\AWeyl p q$. In the complex case, if $q$ and all the coef\/f\/icients of $p$ are real, $\AWeyl p q$ can be made into a $*$-algebra and $(\Omega, d)$ into a $*$-calculus. In the case of the linear polynomial $p(z) = 1-z$ the $*$-calculus $(\Omega, d)$ coincides with the 3D (left) covariant calculus on ${\rm SU}_q(2)$ introduced by Woronowicz in~\cite{Wor:twi}. Since the bimodule of one-forms~$\Omega$ is free both as a left and right $\AWeyl p q$-module, the cotangent bundle of the noncommutative space described by $\AWeyl p q$ is trivial. Following the procedure described in~\cite{BrzElK:int} we associate a~divergence or a~hom-connection~\cite{Brz:con} to $(\Omega, d)$, establish that the dimension of its cokernel or the integral space is equal to the degree of~$p(z)$ and f\/ind a recursive formula that determines the associated integral.

Next we discuss the restriction of $(\Omega, d)$ to the calculus on the degree-zero subalgebra of~$\AWeyl p q$ isomorphic to $\Weyl p q$. We show that in the regular case, the sections of the cotangent bundle are isomorphic to the direct sum of degree~$-2$ and~$2$ components of $\AWeyl p q$ provided $q^4\neq 1$ or $p$ is linear. This interpretation gives a natural splitting of the calculus on $\Weyl p q$ into holomorphic and anti-holomorphic forms. Furthermore, it allows one to apply techniques developed by Beggs and Majid in~\cite{BegMaj:spe} to construct Dirac operators on $\Weyl p q$. The sections of the spinor bundle are isomorphic to the direct sum of degree $-1$ and $1$ components of~$\AWeyl p q$. We show that this direct sum is free as a left (or right) $\Weyl p q$-module, a result that mimics exactly what happens in the case of the standard quantum 2-sphere~\cite{Maj:Rie}. Again following the Beggs--Majid procedure from~\cite{BegMaj:spe} we equip thus constructed Dirac operators with real structures, and determine their ${\rm KO}$-dimension to be~2 (modulo~8). This further enforces heuristic understanding of regular generalized Weyl algebras $\Weyl p q$ as coordinate algebras of noncommutative Riemann surfaces.

Throughout, all algebras are associative, unital and def\/ined over a f\/ield ${\mathbb K}$ of characteris\-tic~0.

\section{Preliminaries}\label{sec.prem}
\begin{Definition}\label{def.generic}
Let ${\mathbb K}$ be a f\/ield of characteristic 0, $q\in {\mathbb K}$, and let $p(z)$ be a polynomial in one variable with coef\/f\/icients from ${\mathbb K}$. Let
\begin{gather}\label{q-derivative}
p_q(z) := \frac{p(qz) -p(z)}{(q -1)z},
\end{gather}
denote the $q$-derivative of $p$. We say that $p$ is a {\em $q$-separable polynomial} if $p(z)$ is coprime with~$zp_q(z)$.
\end{Definition}

It is clear from Def\/inition~\ref{def.generic} that~0 cannot be a root of a $q$-separable polynomial. Furthermore, the properties of greatest common divisors imply that if $q\neq1$, a non-constant polynomial~$p(z)$ is $q$-separable if and only if it has no common factors with $p(qz)$ other than scalars. In case of $q=1$ in which the $q$-derivative~\eqref{q-derivative} is the usual formal derivative of polynomials, a~$1$-separable polynomial is simply a separable polynomial with a non-zero constant term.

In the following def\/inition the algebras studied in this paper are introduced.

\begin{Definition}\label{def.Weyl}
Let ${\mathbb K}$ be a f\/ield, $q\neq 0$ an element of ${\mathbb K}$, and let $p$ be a polynomial in one variable with coef\/f\/icients from ${\mathbb K}$.

\begin{enumerate}\itemsep=0pt
\item[(1)] The algebra $\AWeyl p q$ is an af\/f\/ine algebra generated by $x_+$, $x_-$, $z_-$, $z_+$ subject to relations:
\begin{subequations}\label{a.weyl}
\begin{gather}
z_+ z_- = z_-z_+, \qquad x_+z_{\pm} = q^{-1}z_\pm x_+ , \qquad x_-z_{\pm} = qz_\pm x_-, \label{a.weyl1}\\
x_+x_- = p(z_+z_-), \qquad x_-x_+ = p\big(q^2z_-z_+\big). \label{a.weyl2}
\end{gather}
\end{subequations}

\item[(2)] The algebra $\Weyl p q$ is an af\/f\/ine algebra generated by $x$, $y$, $z$ subject to relations:
\begin{gather*}
 xz = q^2 z x , \qquad yz = q^{-2}zy, \qquad xy = q^2 zp\big(q^2z\big), \qquad yx = zp(z).
\end{gather*}
\item[(3)] We say that $\AWeyl p q$ (resp.\ $\Weyl p q$) is {\em regular} provided~$p$ is a $q^2$-separable polynomial.
\end{enumerate}
\end{Definition}

Both $\AWeyl p q$ and $\Weyl p q$ are degree-one generalized Weyl algebras \cite{Bav:gweyl}; $\AWeyl p q$ is def\/ined over the ring ${\mathbb K}[z_+, z_-]$, while $\Weyl p q$ extends~${\mathbb K}[z]$.

The algebra $\AWeyl p q$ is a ${\mathbb Z}$-graded algebra with the grading given on the generators by $|z_\pm | =|x_\pm| = \pm 1$. The degree-zero subalgebra ${\AWeyl p q}_0$ is isomorphic to $\Weyl p q$ by the map
\begin{gather}\label{theta}
\Theta\colon \ \Weyl p q \to {\AWeyl p q}_0, \qquad x\mapsto x_-z_+, \qquad y\mapsto z_-x_+, \qquad z\mapsto z_-z_+.
\end{gather}
If ${\mathbb K} ={\mathbb C}$, $q$ is a real number and $p$ has real coef\/f\/icients, then both $\AWeyl p q$ and $\Weyl p q$ can be made into $*$-algebras by setting
\begin{gather}\label{star}
z_-^* = z_+, \qquad x_-^* = x_+; \qquad z^*=z, \qquad x^* =y.
\end{gather}
The map $\Theta$ is then an isomorphism of $*$-algebras. In the case of the ($q^2$-separable) polynomial
\begin{gather*}
p(z) = \prod_{k=0}^{l-1}\big(1- q^{-2\frac{k}l}z\big),
\end{gather*}
the $*$-algebra $\AWeyl p q$ is the coordinate algebra of a quantum lens space~\cite{HonSzy:len}, while $\Weyl p q$ is the coordinate algebra of the quantum teardrop \cite{BrzFai:tea}. Furthermore, if $l=1$, $\AWeyl p q$ is the coordinate algebra of the quantum ${\rm SU}_q(2)$ group \cite{Wor:com}, while $\Weyl p q$ is the standard Podle\'s sphere algebra~\cite{Pod:sph}. Finally, if $p$ is a non-zero constant polynomial, then $x_+$ is invertible with the inverse proportional to $x_-$, hence $\AWeyl p q$ is a $q$-skew Laurent polynomial algebra over the ring of polynomials in two variables. In this case $\Weyl p q$ is the quantum plane (the quantum polynomial ring in two variables).

As explained in \cite[Theorem~3.7]{Brz:gen}, if $p(0)\neq 0$, then $\AWeyl p q$ is a strongly graded algebra, i.e., every homogeneous element of $\AWeyl p q$ of degree $k+l$ can be written as a~linear combination of products of elements of degrees~$k$ and~$l$~\cite{Dad:gro,NasVan:gra}. Equivalently, there exists a mapping (called a~{\em strong connection}~\cite{Haj:str})
\begin{gather*}
\ell \colon \ {\mathbb Z} \to \AWeyl p q \otimes \AWeyl p q,
\end{gather*}
such that, for all $n\in {\mathbb Z}$, $\ell(n) \in \AWeyl p q_{-n} \otimes \AWeyl p q_n$ and the composition of $\ell$ with the multiplication in $\AWeyl p q$ is the constant function $n\mapsto 1 \in \AWeyl p q$. In consequence, all the homogeneous components $\AWeyl p q_{n}$ are projective (in fact invertible) modules over the degree-zero subalgebra identif\/ied with~$\Weyl p q$. These modules are mutually non-isomorphic, provided $q$ is not a root of unity and $p$ has at least one non-zero root. Geometrically all this means that~$\AWeyl p q$ is a~noncommutative principal circle bundle over $\Weyl p q$ and that $\AWeyl p q_{n}$ are (modules of sections of) non-trivial and mutually non-equivalent line bundles over $\Weyl p q$~\cite{BrzMaj:gau}.

\section[Skew derivations on ${\mathcal A}(p;q)$]{Skew derivations on $\boldsymbol{\AWeyl p q}$}\label{sec.skew}
Let ${\mathcal{A}}$ be an algebra and $\sigma$ an algebra automorphism of ${\mathcal{A}}$. We say that a linear map $\partial\colon {\mathcal{A}} \to {\mathcal{A}}$ is a {\em $($right$)$ skew $\sigma$-derivation} or that $(\partial, \sigma)$ is a {\em $($right$)$ skew derivation} if, for all $a,b\in {\mathcal{A}}$,
\begin{gather*}
\partial(ab) = \partial(a) \sigma(b) + a\partial(b).
\end{gather*}

\begin{Theorem}\label{thm.skew}
Let $\AWeyl p q$ be a generalized Weyl algebra as in Definition~{\rm \ref{def.Weyl}}, and let $z:= z_+z_-$. Consider linear maps $\sigma_\pm$, $\sigma_0$ of $\AWeyl p q$ defined on the homogeneous elements $a\in \AWeyl p q$ by
\begin{gather}\label{sigma}
\sigma_\pm(a) = q^{|a|}a, \qquad \sigma_0(a) = q^{2|a|}a,
\end{gather}
and the polynomial
\begin{gather}\label{c}
c(z) : = q p _{q^2}(z) = q\frac{p\big(q^2z\big) -p(z)}{\big(q^2 -1\big)z}.
\end{gather}
Then:
\begin{enumerate}\itemsep=0pt
\item[$(1)$] The maps $\sigma_\pm$, $\sigma_0$ are algebra automorphisms of
$\AWeyl p q$.
\item[$(2)$] For all $\alpha_0 \in {\mathbb K}$, the map $\partial_0$ defined on the generators of $\AWeyl p q$ by
\begin{gather*}
\partial_0(x_+) = \alpha_0 x_+, \qquad \partial_0(x_-) = -q^{-2}\alpha_0 x_-\qquad \partial_0(z_+) = \alpha_0 z_+, \\ \partial_0(z_-) = -q^{-2}\alpha_0 z_-,
\end{gather*}
extends to the whole of $\AWeyl p q$ as a skew $\sigma_0$-derivation.
\item[$(3)$] For all $\alpha_- \in {\mathbb K}$, the map $\partial_-$ defined on the generators of $\AWeyl p q$ by
\begin{gather*}
\partial_-(x_+)=\partial_-(z_+) = 0, \qquad \partial_-(x_-) = \alpha_- c(z) z_+,\qquad \partial_-(z_-) = \alpha_- x_+,
\end{gather*}
extends to the whole of $\AWeyl p q$ as a skew $\sigma_-$-derivation.
\item[$(4)$] For all $\alpha_+ \in {\mathbb K}$, the map $\partial_+$ defined on the generators of~$\AWeyl p q$ by
\begin{gather*}
\partial_+(x_-)=\partial_+(z_-) = 0, \qquad \partial_+(x_+) = \alpha_+ c(z) z_-,\qquad \partial_+(z_+) = \alpha_+ x_-,
\end{gather*}
extends to the whole of $\AWeyl p q$ as a skew $\sigma_+$-derivation.
\item[$(5)$] The maps $\partial_0$, $\partial_{\pm}$ are $q$-skew derivations, i.e.,
\begin{gather}\label{q}
\sigma_0^{-1}\circ \partial_0 \circ \sigma_0 = \partial_0, \qquad \sigma_\pm^{-1}\circ \partial_\pm \circ \sigma_\pm = q^{\pm 2} \partial_\pm.
\end{gather}
\item[$(6)$] The maps $\delta_\pm := \partial_\pm\circ \Theta\colon \Weyl p q \to \AWeyl p q_ {\mp 2}$, where $\Theta$ is the isomorphism~\eqref{theta}, are $\AWeyl p q_{\mp 2}$-valued derivations, while $\partial_0\circ \Theta =0$.
\end{enumerate}
\end{Theorem}
\begin{proof}
(1) The statement follows immediately from the fact that $\AWeyl p q$ is a graded algebra.

(2) We need to check that the skew derivation property of $\partial_0$ is compatible with rela\-tions~\eqref{a.weyl}. We compute
\begin{align*}
\partial_0(z_-z_+) &= \partial_0(z_-)\sigma_0(z_+) + z_- \partial_0(z_+) = \alpha_0\big({-}q^{-2} \times q^2 + 1\big)z_- z_+ = 0\\
&= \alpha_0\big(q^{-2} - q^{-2}\big) z_+ z_- = \partial_0(z_+)\sigma_0(z_-) + z_+ \partial_0(z_-) =\partial_0(z_+z_-).
\end{align*}
This calculation implies further that, for all polynomials $f$, $\partial_0(f(z)) =0$. Hence
\begin{gather*}
\partial_0 (x_+x_-) = \partial_0(x_+)\sigma_0(x_-) + x_+ \partial_0(x_-) = \alpha_0\big(q^{-2} - q^{-2}\big) x_+ x_- = 0 = \partial_0 (p(z))
\end{gather*}
and
\begin{gather*}
\partial_0(x_-x_+) = \partial_0(x_-)\sigma_0(x_+) + x_- \partial_0(x_+) = \alpha_0\big({-}q^{-2} \times q^2 + 1\big)x_- x_+ = 0 = \partial_0 \big(p\big(q^2z\big)\big).
\end{gather*}
In a similar way one easily f\/inds that
\begin{gather*}
\partial_0(x_-z_+) = 0 = \partial_0(qz_+x_-), \qquad \partial_0(z_-x_+) = 0 = \partial_0(qx_+z_-).
\end{gather*}
Finally,
\begin{gather*}
\partial_0(x_+z_+) = \partial_0(x_+)\sigma_0(z_+) + x_+ \partial_0(z_+) = \alpha_0\big( q^2 + 1\big)x_+ z_+
\end{gather*}
and
\begin{gather*}
\partial_0(z_+x_+) = \partial_0(z_+)\sigma_0(x_+) + z_+ \partial_0(x_+) = \alpha_0\big( q^2 + 1\big)z_+ x_+\\
 \hphantom{\partial_0(z_+x_+)}{} = q\alpha_0\big( q^2 + 1\big)x_+ z_+ = \partial_0(qx_+z_+) .
\end{gather*}
In a similar way the compatibility of the $\sigma_0$-skew derivation property of $\partial_0$ with the last of commutation relations in $\AWeyl p q$, i.e., that $\partial_0(x_-z_-) = \partial_0(q z_-x_-)$ is checked.

(3) As in (2) we need to check that the skew derivation property of $\partial_-$ is compatible with relations \eqref{a.weyl}. We start with
\begin{align*}
\partial_-(x_-z_-) &= \partial_-(x_-)\sigma_-(z_-) + x_- \partial_-(z_-) = \alpha_-\big(q^{-1}c(z)z_+z_- +x_-x_+\big)\\
&= \alpha_-\left(\frac{p\big(q^2z\big) -p(z)}{q^2 -1} +p\big(q^2z\big)\right)= \alpha_-\frac{q^2p\big(q^2z\big) -p(z)}{q^2 -1}.
\end{align*}
On the other hand
\begin{align*}
\partial_-(qz_-x_-) &= q \partial_-(z_-)\sigma_-(x_-) + qz_- \partial_-(x_-) = \alpha_-(x_+x_- + qz_-c(z)z_+)\\
&= \alpha_-\left(p(z) + q^2\frac{p\big(q^2z\big) -p(z)}{q^2 -1} \right) = \alpha_-\frac{q^2p\big(q^2z\big) -p(z)}{q^2 -1} = \partial_-(x_-z_-),
\end{align*}
as required. Clearly,
\begin{gather*}
\partial_-(qx_+z_+) = 0 = \partial_-(z_+x_+).
\end{gather*}
 Next
\begin{align*}
\partial_-(z_-z_+) &= \partial_-(z_-)\sigma_-(z_+) + z_- \partial_-(z_+) = q\alpha_-x_+z_+ = \alpha_-z_+x_+\\
& = \partial_-(z_+)\sigma_-(z_-) + z_+ \partial_-(z_-) = \partial_-(z_+z_-).
\end{align*}
In a similar way one f\/inds
\begin{gather*}
\partial_-(x_-z_+) = q\alpha_-c(z) z_+^2 = \partial_-(qz_+x_-) \qquad \mbox{and} \qquad \partial_-(z_-x_+) = q\alpha_- x_+^2 = \partial_-(qx_+z_-) .
\end{gather*}
Finally, observe that, for all polynomials~$f$,
\begin{gather*}
\partial_-(f(z_-)) = \alpha_-\frac{f\big(q^{-2}z_-\big) -f(z_-)}{\big(q^{-2} -1\big)z_-}x_+,
\end{gather*}
and, since $\partial_-(z_+) =0$,
\begin{gather*}
\partial_-(f(z)) = \alpha_-\frac{f\big(q^{-2}z\big) -f(z)}{\big(q^{-2} -1\big)z}z_+x_+.
\end{gather*}
Thus, in particular,
\begin{gather*}
\partial_-(p(z)) = \alpha_- q^{-1}c\big(q^{-2}z\big)z_+x_+ \qquad \mbox{and} \qquad \partial_-\big(p\big(q^2z\big)\big) = \alpha_- q c(z)z_+x_+.
\end{gather*}
Using these observations one easily f\/inds that
\begin{gather*}
\partial_-(p(z)) = \partial_-(x_+x_-) \qquad \mbox{and} \qquad \partial_-\big(p\big(q^2z\big)\big) = \partial_-(x_-x_+),
\end{gather*}
as required.

(4) This statement is proven in a way analogous to statement (3). Alternatively, one can observe that the assignment $x_\pm \mapsto x_\mp$, $z_\pm\mapsto z_\mp$, def\/ines an anti-algebra automorphism $\varphi$ of ${\mathcal{A}}(p;q)$, under which $\partial_+ \propto \varphi^{-1}\circ \partial_-\circ \varphi$; hence (3) implies~(4).

(5) One can easily check that $\partial_0$ is a degree-zero map, while $\partial_\pm$ are degree $\mp 2$ maps, i.e., for all homogeneous $a$,
\begin{gather*}
|\partial_0(a)| = |a|, \qquad |\partial_\pm(a)| = |a|\mp 2.
\end{gather*}
In view of the form of the automorphisms \eqref{sigma}, equations \eqref{q} follow immediately.

(6) Since the automorphisms $\sigma_\pm$ restricted to the degree-zero subalgebra of $\AWeyl p q$ are the identity maps, the restrictions of skew derivations $\partial_\pm$ to this subalgebra satisfy the usual Leibniz rule, and thus so do the $\delta_\pm$ by the algebra map property of~$\Theta$. Furthermore, since $\partial_\pm$ are degree~$\mp 2$ maps, the codomains of $\delta_\pm$ come out as stated. The calculations of part~(2) together with the (skew) derivation property show that $\partial_0$ restricted to the degree-zero subalgebra of~$\AWeyl p q$ is the zero map, and so is~$\delta_0$.
\end{proof}

\section[Dif\/ferential and integral calculi on ${\mathcal A}(p;q)$]{Dif\/ferential and integral calculi on $\boldsymbol{\AWeyl p q}$}\label{sec.diff.a}
In this section, to the system of skew derivations on $\AWeyl p q$ constructed in Theorem~\ref{thm.skew} we associate an $\AWeyl p q$-bimodule $\Omega$ with a derivation $d\colon \AWeyl p q \to \Omega$, and show that $(\Omega, d)$ forms a~f\/irst-order dif\/ferential calculus for $\AWeyl p q$. We then proceed to discuss the canonical divergence or hom-connection associated to $(\Omega, d)$, and calculate the corresponding integral and integral space. We assume that the constants $\alpha_\pm$, $\alpha_0$ in Theorem~\ref{thm.skew} are not equal to zero.

\begin{Definition}\label{def.diff.calc}
Let ${\mathcal{A}}$ be an algebra. By a~{\em first-order differential calculus} on ${\mathcal{A}}$ we mean an ${\mathcal{A}}$-bimodule $\Omega$ together with a ${\mathbb K}$-linear map $d\colon {\mathcal{A}} \to \Omega$ such that
\begin{enumerate}\itemsep=0pt
\item[(a)] $d$ is an $\Omega$-valued derivation of ${\mathcal{A}}$, i.e., it satisf\/ies the Leibniz rule: for all $a,b\in {\mathcal{A}}$,
\begin{gather*}
 d(ab) = d(a) b + ad(b);
\end{gather*}
 \item[(b)] $\Omega$ satisf\/ies the {\em density condition} $\Omega = {\mathcal{A}} d({\mathcal{A}})$, i.e., for all $\omega \in \Omega$, there exists a f\/inite number of elements $a_i,b_i\in {\mathcal{A}}$ such that
 \begin{gather*}
 \omega = \sum_i a_i d(b_i).
 \end{gather*}
 \end{enumerate}

If ${\mathcal{A}}$ is a complex $*$-algebra, then the calculus $(\Omega, d)$ is said to be a {\em $*$-calculus} provided~$\Omega$ is equipped with an anti-linear operation~$*$ such that, for all $a,b\in {\mathcal{A}}$, $\omega\in \Omega$,
\begin{gather*}
(a\omega b)^* = b^*\omega^* a^* \qquad \mbox{and}\qquad d(a^*) = d(a)^*.
\end{gather*}
\end{Definition}

Note that due to the Leibniz rule, the density condition can be equivalently stated as $\Omega = d({\mathcal{A}}) {\mathcal{A}}$ or $\Omega = d({\mathcal{A}}) {\mathcal{A}} +{\mathcal{A}} d({\mathcal{A}})$.

Following the construction of~\cite{BrzElK:int}, f\/ix a f\/inite indexing set $I$, and let $(\partial_i, \sigma_i)$, $i\in I$, be a~collection of (right) skew derivations on an algebra~${\mathcal{A}}$. Let $\Omega$ be a free left ${\mathcal{A}}$-module with a~free basis $\omega_i$, $i\in I$. Def\/ine the (free) right ${\mathcal{A}}$-module structure on~$\Omega$ by setting
\begin{gather}\label{right}
\omega_i a := \sigma_i(a) \omega_i.
\end{gather}
Then the map
\begin{gather}\label{ext}
d\colon \ {\mathcal{A}} \to \Omega, \qquad a\mapsto \sum_{i\in I} \partial_i(a)\omega_i,
\end{gather}
is an $\Omega$-valued derivation of ${\mathcal{A}}$, i.e., it satisf\/ies condition (a) in Def\/inition~\ref{def.diff.calc}. There is no guarantee in general that the density condition of Def\/inition~\ref{def.diff.calc} be satisf\/ied, though.

\begin{Theorem}\label{thm.calc.A}
Let $\partial_0$, $\partial_\pm$ be skew derivations of a regular generalized Weyl algebra $\AWeyl p q$ defined in Theorem~{\rm \ref{thm.skew}}, with $\alpha_0, \alpha_\pm \neq 0$. Set $I = \{-,0,+\}$, let $\Omega$ be a free left $\AWeyl p q$-module with a basis $\omega_-$, $\omega_0$, $\omega_+$, and equip it with the right module structure~\eqref{right}. If $q^2 \neq 1$ or $\deg(p(z))\leq 1$, then $\Omega$ together with the map $d\colon \AWeyl p q\to \Omega$ defined by~\eqref{ext} is a first-order calculus on~$\AWeyl p q$.
\end{Theorem}

\begin{proof}
We start by listing explicitly relations that def\/ine $\Omega$ and $d$ on generators. The bimodule structure of $\Omega$ is determined from the relations:
\begin{subequations}\label{omegas}
\begin{alignat}{3}
& \omega_- z_\pm = q^{\pm 1} z_\pm\omega_-, \qquad && \omega_- x_\pm = q^{\pm 1} x_\pm \omega_-,&\label{omega-}
\\
& \omega_0 z_\pm = q^{\pm 2} z_\pm\omega_0, \qquad && \omega_0 x_\pm = q^{\pm 2} x_\pm\omega_0 ,&\label{omega0}
\\
& \omega_+z_\pm = q^{\pm 1} z_\pm\omega_+, \qquad && \omega_+ x_\pm = q^{\pm 1} x_\pm \omega_+. &\label{omega+}
\end{alignat}
\end{subequations}
The action of $d$ on the generators of $\AWeyl p q$ comes out as
\begin{subequations}\label{d.gen}
\begin{gather}
d(x_+) = \alpha_0 x_+ \omega_0 + \alpha_+c(z)z_- \omega_+, \label{dx+}
\\
d(x_-) = -q^{-2}\alpha_0 x_- \omega_0 + \alpha_-c(z)z_+\omega_-, \label{dx-}
\\
d(z_+) = \alpha_0 z_+ \omega_0 + \alpha_+x_- \omega_+, \label{dz+}
\\
d(z_-) = -q^{-2}\alpha_0 z_- \omega_0 + \alpha_-x_+ \omega_-, \label{dz-}
\end{gather}
\end{subequations}
where, as before, we write $z = z_-z_+$ and $c(z)$ is given by \eqref{c}.
We need to prove that
\begin{gather*}
\omega_\pm, \omega_0\in \AWeyl p q d (\AWeyl p q ) = \AWeyl p q d (\AWeyl p q )+ d (\AWeyl p q ) \AWeyl p q .
\end{gather*}
First assume that $q^2\neq 1$. Using relations \eqref{omegas} and \eqref{d.gen} one f\/inds that
\begin{gather*}
z_-d(z_+) - d(z_+)z_- = \big(1-q^{-2}\big)\alpha_0 z\omega_0,
\end{gather*}
and since $\alpha_0 \neq 0$ and $q^2\neq 1$, we conclude that $z\omega_0 \in \AWeyl p q d (\AWeyl p q )$. Again using rela\-tions~\eqref{omegas} and~\eqref{d.gen} one f\/inds that
\begin{gather*}
d(x_+)x_- - q^{-1}c(z)z_-d(z_+) + q^2 x_+d(x_-) - qz_+ d(z_-)c(z) = \big(q^{-2}-1\big)\alpha_0 p(z)\omega_0,
\end{gather*}
hence also $p(z)\omega_0 \in \AWeyl p q d (\AWeyl p q )$. Since the polynomial $p(z)$ is $q^2$-separable it does not have a~factor~$z$, and thus, by the B\'ezout identity there exist polynomials~$f(z)$,~$g(z)$ such that
\begin{gather*}
f(z)z +g(z)p(z) = 1.
\end{gather*}
In the light of this, multiplying from the left $z\omega_0$ by $f(z)$ and $p(z)\omega_0$ by $g(z)$ we obtain that $\omega_0 \in \AWeyl p q d (\AWeyl p q )$, as required.

If $p(z)= \lambda z +\mu$, with scalar coef\/f\/icients $\lambda$ and $\mu\neq 0$ (which is necessary for $p$ to be $q^2$-regular), then in view of \eqref{d.gen}, \eqref{omegas} and \eqref{a.weyl}
\begin{gather*}
x_- d(x_+) - q^2\lambda z_-d(z_+) = \alpha_0\mu \omega_0,
\end{gather*}
and so $\omega_0 \in \AWeyl p q d (\AWeyl p q )$ with no assumptions on $q$.

In view of equations \eqref{d.gen}, also
\begin{gather*}
c(z)z_\pm \omega_\pm,\ x_\pm\omega_\pm \in \AWeyl p q d (\AWeyl p q ).
\end{gather*}
Multiplying the f\/irst pair by $z_\mp$ and the second by $x_\mp$, and using relations \eqref{a.weyl} we conclude that
\begin{gather*}
zc(z)\omega_\pm,\, p(z)\omega_\pm \in \AWeyl p q d (\AWeyl p q ).
\end{gather*}
Since the polynomial $p(z)$ is $q^2$-separable, there is a polynomial combination of~$p(z)$ and~$zc(z)$ that gives~1, and hence $\omega_\pm \in \AWeyl p q d (\AWeyl p q )$. This completes the proof of the density condition and thus also of the theorem.
\end{proof}

\begin{Remark}\label{rem.star}
If $\AWeyl p q$ is a complex $*$-algebra as in \eqref{star}, then the calculus described in Theorem~\ref{thm.calc.A} can be made into a $*$-calculus with the $*$-operation def\/ined on the generators by
\begin{gather*}
\omega^*_{\pm} = \omega_\mp, \qquad \omega_0^* = - \omega_0,
\end{gather*}
provided $\alpha_0 \in {\mathbb R}$ and $\alpha_-^* = q\alpha_+$.
\end{Remark}
\begin{Remark}\label{3d.wor}
Up to rescaling of generators or up to making suitable choices for the parame\-ters~$\alpha_0$,~$\alpha_\pm$, in the case of $p(z)=1-z$, the calculus in Theorem~\ref{thm.calc.A} coincides with the 3D calculus of Woronowicz on the quantum ${\rm SU}_q(2)$-group~\cite{Wor:twi}.
\end{Remark}

\begin{Remark}\label{rem.classical}
If $\deg(p(z))>1$, the arguments used in the proof of Theorem~\ref{thm.calc.A} make signif\/icant use of the fact that $q^2\neq 1$, hence they do not apply to the classical case $q=1$. However, if~$p(z)$ is a non-zero constant or linear polynomial with a non-zero constant term, such as for example in the preceding remark, the assertion of Theorem~\ref{thm.calc.A} remains true if $q=1$. The resulting commutative algebra is the coordinate algebra of a parallelisable manifold (the product of the circle with the plane or the three-sphere), hence the module of sections of the cotangent bundle is free. If $p$ has higher degree, there is no reason why the corresponding manifold should be a~parallelisable three-manifold, and thus no reason to expect its cotangent bundle to be (globally) trivial. In short, one can thus say that the $q$-deformation trivializes the cotangent bundle in this case or that noncommutativity forces triviality.

Note also that the if $q\neq \pm 1$, the assertions of Theorem~\ref{thm.calc.A} hold also for polynomials with repeated roots (as long as they are $q^2$-separable) as the $q^2$-separability of $p(z)$ does not imply that $p(z)$ has no repeated roots (this implication only holds in the case $q^2=1$). If $p(z)$ has repeated roots then both in the commutative and noncommutative cases, the algebra~$\AWeyl p q $ has inf\/inite global dimension~\cite{Bav:gweyl}, and hence it is not homologically smooth in the sense of~\cite{Van:rel}, i.e., it does not admit a f\/inite-length resolution by f\/initely generated and projective bimodules. Yet, it admits a globally def\/ined f\/irst-order dif\/ferential calculus. In the light of arguments for the interpretation of a homologically smooth algebra as a coordinate algebra on a smooth noncommutative space~\cite{Kra:Hoc} the author f\/inds this quite surprising.
\end{Remark}

The notions recalled in the following def\/inition were introduced in \cite{Brz:con} and \cite{BrzElK:int}.
\begin{Definition}\label{def.div}
Let ${\mathcal{A}}$ be an algebra with a f\/irst-order dif\/ferential calculus $(\Omega, d)$. Let ${\mathfrak I}$ denote the space of right ${\mathcal{A}}$-linear maps $\Omega \to {\mathcal{A}}$ equipped with the left and right ${\mathcal{A}}$-actions:
\begin{gather*}
a\cdot f\cdot b(\omega) := af(b\omega), \qquad \mbox{for all}\quad a,b\in {\mathcal{A}}, \quad \omega\in \Omega,\quad f\in {\mathfrak I}.
\end{gather*}
A {\em divergence} or an {\em ${\mathcal{A}}$-valued hom-connection} on ${\mathcal{A}}$ (relative to the f\/irst-order calculus $(\Omega, d)$) is a~linear map
$\operatorname{div} \colon {\mathfrak I} \to {\mathcal{A}}$, such that, for all $a\in {\mathcal{A}}$ and $\xi\in {\mathfrak I}$,
\begin{gather*}
\operatorname{div}(\xi\cdot a) = \operatorname{div}(\xi) a + \xi(d(a)).
\end{gather*}
The cokernel $\operatorname{coker} (\operatorname{div})$ is called the {\em integral space} associated to $\operatorname{div}$ and the cokernel map $\Lambda\colon {\mathcal{A}} \to \operatorname{coker} (\operatorname{div})$ is called the {\em integral} on~${\mathcal{A}}$ associated to~$\operatorname{div}$. The triple $({\mathfrak I}, \operatorname{div},\Lambda)$ is called a~{\em first-order integral calculus} on~${\mathcal{A}}$.
\end{Definition}

\begin{Proposition}\label{prop.int}
Let $(\Omega, d)$ be the first-order differential calculus on a regular generalized Weyl algebra ${\mathcal{A}}= \AWeyl p q$ with the parameter~$q$ that is not a root of unity, as described in Theorem~{\rm \ref{thm.calc.A}}. Let ${\mathfrak I} = \rhom {\mathcal{A}} \Omega {\mathcal{A}}$ and set $n = \deg(p(z))$. Define
\begin{gather}\label{div.A}
\operatorname{div} \colon \ {\mathfrak I} \to {\mathcal{A}}, \qquad \xi \mapsto q^{-2}\partial_-(\xi(\omega_-)) + \partial_0(\xi(\omega_0)) + q^{2}\partial_+(\xi(\omega_+)).
\end{gather}
\begin{enumerate}\itemsep=0pt
\item[$(1)$] The map $\operatorname{div}$ is the unique divergence such that $\operatorname{div}(\xi_i)=0$, $i=-,0,+$, where the $\xi_i$ form the basis for ${\mathfrak I}$ dual to the $\omega_i$, i.e., $\xi_i(\omega_j) = \delta_{ij}$, $i,j\in \{-,0,+\}$.
\item[$(2)$]
The integral $\Lambda$ vanishes on all elements of $\AWeyl p q$ except for the subalgebra ${\mathbb K}[z]$, where it satisfies, for all $f(z) \in {\mathbb K}[z]$,
\begin{gather}\label{lambda.A}
\Lambda\big(q^2 p\big(q^2 z\big)f\big(q^2 z\big)\big) = \Lambda (p(z)f(z) ).
\end{gather}
\item[$(3)$] The integral space associated to $\operatorname{div}$ is $n$-dimensional with a basis $v_i := \Lambda(z^i)$, $i=0,1,\ldots,$ $n-1$. Furthermore, for all $k\in {\mathbb N}$,
\begin{gather}\label{lambda.compact}
\Lambda \big(z^{n+k}\big) = \sum_{i=0}^{n-1} \frac{[i+1]}{[n+k+1]}\beta^k_i v_i,
\end{gather}
where
\begin{gather}\label{q-int}
[l] := \frac{1-q^{2l}}{1-q^2} = 1+ q^2 +\cdots + q^{2l-2}, \qquad l\in {\mathbb N},
\end{gather}
denote $q^2$-integers, and the coefficients $\beta_i^k$ are determined from the recurrence relations as follows. Let $\hat{p}(z)$ be the monic polynomial associated to $p(z)$ and write
\begin{gather}\label{hatp}
\hat{p}(z) = z^n - \sum_{i=0}^{n-1} \mu_i z^i.
\end{gather}
Then
\begin{gather}\label{beta.rec}
\beta_i^k = \sum_{j=1}^n \mu_{n-j}\beta_i^{k-j} +\mu_{i-k}, \qquad \mu_l = \beta^l_i=0 \quad \mbox{if} \quad l<0.
\end{gather}
\end{enumerate}
\end{Proposition}
\begin{proof}
(1) Since, by Theorem~\ref{thm.skew}, the calculus $(\Omega ,d)$ is associated to $q$-skew derivations $\partial_i$, the assertion follows from \cite[Theorem~3.4]{BrzElK:int}.

(2) Using the def\/inition of $\partial_0$ and the fact that it is a skew derivation one easily f\/inds that, for all $m >0$,
\begin{gather*}
\partial_0\big(x_+^n\big) = \alpha_0 [n] x_+^n, \qquad \partial_0\big(z_+^n\big) = \alpha_0 [n] z_+^n,
\end{gather*}
where the notation \eqref{q-int} is used, and hence
\begin{gather}\label{0+}
\partial_0\big(x_+^nz_+^m\big) = \alpha_0 [m+n] x_+^nz_+^m.
\end{gather}
In a similar way,
\begin{gather}\label{0-}
\partial_0\big(x_-^nz_-^m\big) = \alpha_0 q^{-2(m+n-1)} [m+n] x_-^nz_-^m.
\end{gather}
A linear basis for $\AWeyl p q$ consists of $x_-^nz_+^mz_-^l$, $x_+^nz_+^mz_-^l$. Since the degree-zero subalgebra~${\mathcal{B}}$ of~$\AWeyl p q$ is isomorphic to $\Weyl p q$, and hence it is generated by $x=x_-z_+$, $y = z_-x_+$, $z=z_-z_+$, for all positive~$k$, the degree $k$-part of $\AWeyl p q$ is generated by
\begin{gather*}
x_+^k, \,x_+^{k-1}z_+,\, \ldots ,\, x_+z_+^{k-1},\, z_+^k,
\end{gather*}
as a ${\mathcal{B}}$-module. Similarly, for a negative $k$, as a ${\mathcal{B}}$-module, the degree $k$ component of~$\AWeyl p q$ is generated by
\begin{gather*}
x_-^{-k},\, x_-^{-k-1}z_-, \,\ldots ,\, x_-z_-^{-k-1},\, z_-^{-k}.
\end{gather*}
By statement (6) of Theorem~\ref{thm.skew}, $\partial_0$ restricted to ${\mathcal{B}}$ is the zero map, and thus, being a skew derivation twisted by an automorphism that is the identity on~${\mathcal{B}}$, it is a ${\mathcal{B}}$-bimodule map. Therefore, in view of formulae~\eqref{0+},~\eqref{0-} and~\eqref{div.A},
\begin{gather}\label{not0}
\bigoplus_{k\in {\mathbb Z}\setminus \{0\}} \AWeyl p q _k \subseteq \operatorname{Im} (\partial_0 )\subseteq \operatorname{Im} (\operatorname{div} )= \ker \Lambda.
\end{gather}

Next, using the def\/inition of $\partial_+$ and the fact that it is a skew derivation, one easily f\/inds that, for all $n>0$,
\begin{gather*}
\partial_+\big(z_+^n\big) = \alpha_+ q^{-n+1} [n] x_-z_+^{n-1}.
\end{gather*}
Since $\partial_+(x_-) = \partial_+(z_-) = 0$, we can compute
\begin{align*}
\partial_+\big(x_-^kz_{+}^{k+l+2}z_-^l\big) &= q^{-l} x_-^k\partial_+\big(z_{+}^{k+l+2}\big)z_-^l = \alpha_+
q^{-k-2l-1}[k+l+2] x_-^{k+1}z_+^{k+1} z_+^lz_-^l\\
&= \alpha_+ q^{-2l +\frac{(k-1)k}2 - 1}[k+l+2] x^{k+1} z^l.
\end{align*}
Therefore, for all $k,l\in {\mathbb N}$,
\begin{gather}\label{notxz}
x^{k+1} z^l \in \operatorname{Im} (\partial_+ )\subseteq \operatorname{Im} (\operatorname{div} ) = \ker \Lambda.
\end{gather}
Arguing in a similar way but with $\partial_-$ instead of $\partial_+$, one f\/inds that
\begin{gather}\label{notyz}
y^{k+1} z^l \in \operatorname{Im}\left(\partial_-\right)\subseteq \operatorname{Im}\left(\operatorname{div}\right) = \ker \Lambda.
\end{gather}
Putting \eqref{not0}, \eqref{notxz} and \eqref{notyz} together we can conclude that $\Lambda$ could be non-zero only on polynomials in~$z$.

In view of the def\/initions of the divergence~\eqref{div.A} and skew derivations~$\partial_\pm$, the only way to obtain a polynomial in $z$ in the image of $\operatorname{div}$ is to apply $\partial_-$ to $x_-z_- f(z)$ or $\partial_+$ to $x_+z_+ g(z)$, where $f(z),g(z)\in {\mathbb K}[z]$. This gives
\begin{gather*}
\partial_-(x_-z_- f(z)) = \frac{\alpha_-}{q^2 -1}\big(q^2 p(q^2 z)f(q^2 z) - p(z)f(z)\big)
\end{gather*}
and
\begin{gather*}
\partial_+(x_+z_+ g(z)) = \frac{\alpha_+}{q^2 -1}\big(q^2 p\big(q^2 z\big)g(z) - p(z)g\big(q^{-2}z\big)\big).
\end{gather*}
Therefore, the only polynomials in $z$ contained in $\operatorname{Im}(\operatorname{div}) = \ker \Lambda$ have the form
\begin{gather*}
q^2 p\big(q^2 z\big)f\big(q^2 z\big) - p(z)f(z), \qquad f(z)\in {\mathbb K}[z].
\end{gather*}
This proves the equation \eqref{lambda.A}.

(3) Equation \eqref{lambda.A} gives a recurrence relation for the values of $\Lambda$ on powers of~$z$. Clearly,~\eqref{lambda.A} remains true if $p$ is replaced by the associated monic $\hat{p}$, and when the latter is written as in~\eqref{hatp}, then
the formula~\eqref{lambda.A}, evaluated at $f(z)= z^k$, $k\in {\mathbb N}$, gives
\begin{gather}
\Lambda\big(z^{n+k}\big) = \frac{1}{q^{2n +2k +2} -1}\sum_{i=0}^{n-1}\big(q^{2k +2i +2} -1\big) \mu_i\Lambda\big(z^{k+i}\big)
 = \sum_{i=0}^{n-1} \frac{[i+k+1]}{[n+k+1]}\mu_i \Lambda\big(z^{k+i}\big).\label{lambda.A.rec}
\end{gather}
As the recurrence relation \eqref{lambda.A.rec} has order~$n$ and $\mu_0\neq 0$, its space of solutions, which coincides with the integral space for~$\Lambda$, is $n$-dimensional as stated. Its basis can be chosen as
\begin{gather*}
v_0 = \Lambda(1),\, v_1 = \Lambda(z), \,\ldots ,\, v_{n-1} = \Lambda\big(z^{n-1}\big).
\end{gather*}
Using \eqref{lambda.A.rec} repeatedly, and observing that $\Lambda(z^{k+i})$ has the $q^2$-integer $[i+k+1]$ in the de\-no\-mi\-nator, which cancels out $[i+k+1]$ in~\eqref{lambda.A.rec} one realizes that the expression for $\Lambda(z^{n+k})$ in terms of the $v_i$ depends on $q^2$-rational numbers precisely as in~\eqref{lambda.compact}. Then the formula~\eqref{beta.rec} for $q$-independent coef\/f\/icients is obtained and verif\/ied inductively.
\end{proof}

\begin{Example}\label{ex.lambda}
Let us consider a regular generalized Weyl algebra $\AWeyl p q$ with~$q$ not a root of unity and
\begin{gather*}
p(z) = z^2 -1.
\end{gather*}
Then $\mu_0 = 1$ and $\mu_1 = 0$ and the recurrence relations \eqref{beta.rec} take the form
\begin{gather*}
\beta^0_0 = \beta^1_1 = 1, \qquad \beta^0_1 = \beta^1_0 = 0, \qquad \beta^k_i = \beta^{k-2}_i, \qquad \mbox{for all} \quad k>1.
\end{gather*}
Thus, by \eqref{lambda.compact} the integral associated to the calculus $\Omega$ is
\begin{gather*}
\Lambda \big(z^k\big) =
\begin{cases} \dfrac{1}{[k +1]} \Lambda (1), & \mbox{if $k$ is even},\vspace{1mm}\\
\dfrac{[2]}{[k +1]} \Lambda (z), & \mbox{if $k$ is odd},
\end{cases}
\end{gather*}
and zero on all other elements of the basis $x_-^nz_+^mz_-^l$, $x_+^nz_+^mz_-^l$ for $\AWeyl p q$.
\end{Example}

\begin{Example}\label{ex.lambda.2}
Let us consider a regular generalized Weyl algebra $\AWeyl p q$ with $q$ not a root of unity and
\begin{gather*}
p(z) = (z -1)^2.
\end{gather*}
Then $\mu_0 = -1$ and $\mu_1 = 2$ and the recurrence relations \eqref{beta.rec} take the form
\begin{gather*}
\beta^0_0 = -1, \!\qquad \beta^0_1 = 2 , \!\qquad \beta^1_0 = -2, \!\qquad \beta^1_1 = 3, \!\qquad \beta^k_i = 2\beta^{k-1}_i - \beta^{k-2}_i, \!\qquad \mbox{for all}
\quad k>1,
\end{gather*}
which have solutions
\begin{gather*}
\beta^k_0 = -(k+1), \qquad \beta^k_1 = k+2.
\end{gather*}
Therefore,
\begin{gather*}
\Lambda \big(z^k\big) = -\frac{k-1}{[k +1]} \Lambda (1)+ \frac{[2] }{[k +1]}k \Lambda (z),
\end{gather*}
and zero on all other elements of the basis $x_-^nz_+^mz_-^l$, $x_+^nz_+^mz_-^l$ for $\AWeyl p q$.
\end{Example}

\section[Dif\/ferential calculus and Dirac operators on ${\mathcal B}(p;q)$]{Dif\/ferential calculus and Dirac operators on $\Weyl p q$}\label{sec.diff.b}
In this section we identify the algebra $\Weyl p q$ with the degree-zero subalgebra of $\AWeyl p q$ via the isomorphism \eqref{theta}, and describe a f\/irst-order dif\/ferential calculus and Dirac operators on~$\Weyl p q$. We assume that the parameters $\alpha_\pm$ in Theorem~\ref{thm.skew} are non-zero.

\begin{Theorem}\label{thm.calc.B}
Let $\AWeyl p q$ be a regular generalized Weyl algebra, with $q$ not a quartic root of unity or otherwise $\deg(p(z)) \leq 1$, and let $(\Omega ,d)$ be the first-order differential calculus defined in Theorem~{\rm \ref{thm.calc.A}}. View $\Weyl p q$ as a degree-zero subalgebra of $\AWeyl p q$ and let
\begin{gather*}
\bar\Omega := \Weyl p q d (\Weyl p q ) \Weyl p q
\end{gather*}
be the calculus on $\Weyl p q$ induced from $(\Omega ,d)$. Then
\begin{gather}\label{bomega}
\bar\Omega \cong \AWeyl p q_{2} \oplus \AWeyl p q_{-2},
\end{gather}
as $\Weyl p q$-bimodules.
\end{Theorem}
\begin{proof}
Since $\omega_\pm$ are generators of degrees $\pm 2$, commuting with elements of $\Weyl p q$, and $\Weyl p q$ $\subseteq \ker(\partial_0)$,
\begin{gather*}
\bar\Omega \subseteq \AWeyl p q_{2}\,\omega_- \oplus \AWeyl p q_{-2}\,\omega_+.
\end{gather*}

As a $\Weyl p q$-module, $ \AWeyl p q_{2}$ is generated by $z_+^2$, $x_+^2$, $z_+x_+$ while $ \AWeyl p q_{-2}$ is generated by~$z_-^2$, $x_-^2$, $z_-x_-$, hence to show that the above inclusion is an equality, it suf\/f\/ices to prove that
\begin{gather*}
z_+^2\omega_- ,\, x_+^2\omega_-,\, z_+x_+\omega_-,\,x_-^2 \omega_+,\, z_-^2\omega_+,\, z_-x_- \omega_+ \in \bar\Omega.
\end{gather*}

First we consider the case $q^4\neq 1$. Using the def\/inition of $d$ in terms of skew derivations~$\partial_\pm$ and~$\partial_0$ or using equations~\eqref{d.gen} one easily f\/inds that
\begin{subequations}\label{d.gen.B}
\begin{gather}
d(x) = q\alpha_-c(z)z_+^2\omega_- + \alpha_+x_-^2 \omega_+, \label{dx}
\\
d(y) = q\alpha_-x_+^2\omega_- + \alpha_+c(z)z_-^2\omega_+ , \label{dy}
\\
d(z) = \alpha_-z_+x_+\omega_- + \alpha_+z_-x_- \omega_+. \label{dz}
\end{gather}
\end{subequations}
Therefore, since the generators $\omega_\pm$ commute with all elements of $\Weyl p q$, we conclude that
\begin{subequations}\label{xdz+zdx}
\begin{gather}
zd(x) - d(x)z = \big(1-q^4\big)\alpha_+ zx_-^2 \omega_+, \label{zdx}
\\
zd(x) = q\alpha_-zc(z)z_+^2\omega_- + \alpha_+zx_-^2 \omega_+, \label{xdz}
\\
d(z)x = \alpha_-p(z)z_+^2\omega_- + q^2\alpha_+zx_-^2 \omega_+. \label{dzx}
\end{gather}
\end{subequations}
Equation \eqref{zdx} implies that $zx_-^2 \omega_+\in \bar\Omega$, hence also $zc(z)z_+^2\omega_- , p(z)z_+^2\omega_- \in \bar\Omega$, by~\eqref{xdz} and~\eqref{dzx}. Since $p(z)$ is $q^2$-separable, the latter implies that $z_+^2\omega_- \in \bar\Omega$ and, in view of~\eqref{dx}, $x_-^2\omega_+\in \bar\Omega$. In a similar way one f\/inds that
\begin{subequations}\label{ydz+zdy}
\begin{gather}
zd(y) - d(y)z = \big(q-q^{-3}\big)\alpha_- zx_+^2 \omega_-, \label{zdy}
\\
yd(z) = q^{-1}\alpha_-zx_+^2\omega_- + q^{-1}\alpha_+p(z)z_-^2 \omega_+, \label{ydz}
\\
zd(y) = q\alpha_-zx_+^2\omega_- + \alpha_+zc(z)z_-^2 \omega_+. \label{dzy}
\end{gather}
\end{subequations}
By \eqref{zdy}, $zx_+^2 \omega_-\in \bar\Omega$, hence also $zc(z)z_-^2\omega_+, p(z)z_-^2\omega_+ \in \bar\Omega$, by \eqref{ydz} and \eqref{dzy}. Since $p(z)$ is $q^2$-separable, the latter implies that $z_-^2\omega_+ \in \bar\Omega$ and, in view of \eqref{dy}, $x_+^2\omega_-\in \bar\Omega$. Finally,
\begin{gather*}
\alpha_+ p(z)z_-x_-\omega_+ = yd(x) - \frac{q^{-1}c\big(q^{-2}z\big)}{1-q^{-4}}\big(zd(z) - q^{-2}d(z)z\big),
\end{gather*}
and
\begin{gather*}
\alpha_+ zz_-x_-\omega_+ =\frac{1}{1-q^{4}}\big(zd(z) - q^{2}d(z)z\big),
\end{gather*}
hence $z_-x_-\omega_+ \in \bar \Omega$, since the polynomial $p(z)$, being $q^2$-separable, does not contain the factor~$z$. In view of~\eqref{dz},
also $z_+x_+\omega_- \in \bar \Omega$.

If $p(z) = \lambda z +\mu$, $\lambda, \mu\in {\mathbb K}$, $\mu\neq 0$, then equations~\eqref{d.gen.B} and the fact that generators $\omega_\pm$ commute with all elements of $\Weyl p q$ yield
\begin{gather*}
\alpha_-\mu z_+^2\omega_- = d(z) x - q^{-2}d(x)z, \qquad \alpha_+\mu z_-^2\omega_+ = qyd(z) - q^{-1}z d(y)
\end{gather*}
and
\begin{gather*}
\alpha_-\mu z_+x_+\omega_- = q^{-2}x d(y) - q^2 \lambda z d(z),
\end{gather*}
hence also in this case all the required forms are in $\bar{\Omega}$.
Since $\omega_\pm$ are ($\Weyl p q$-central) free generators, we obtain the required isomorphism of $\Weyl p q$-bimodules \eqref{bomega}.
\end{proof}

Theorem~\ref{thm.calc.B} shows that the $\Weyl p q$-bimodule $\AWeyl p q_{2} \oplus \AWeyl p q_{-2}$ plays the role of sections of the cotangent bundle over the quantum surface with the coordinate algebra $\Weyl p q$. The direct sum decomposition on the right hand side of \eqref{bomega} can be interpreted as the decomposition of one-forms into holomorphic and anti-holomorphic parts. This is particularly justif\/ied in case~$\AWeyl p q$ and $\Weyl p q$ are equipped with the $*$-algebra structure~\eqref{star} and~$\Omega$ (and hence also~$\bar\Omega$) is a~$*$-calculus as in Remark~\ref{rem.star}. Following this, we write
\begin{gather*}
\bar\Omega^{1,0} = \AWeyl p q _{-2} \omega_+, \qquad \bar\Omega^{0,1} = \AWeyl p q _{2} \omega_-.
\end{gather*}
Theorem~\ref{thm.calc.B} shows further that the module of horizontal forms
\begin{gather*}
\Omega_{\mathrm{hor}} := \AWeyl p q d (\Weyl p q ) \AWeyl p q,
\end{gather*}
of the principal circle bundle $\AWeyl p q$ over $\Weyl p q$ is generated by $\omega_\pm$. The one-form $\omega_0$ generates vertical forms. We write $\pi$ for the projection of $\Omega$ onto $\Omega_{\mathrm{hor}}$. In particular, for all $a\in \AWeyl p q$,
\begin{gather*}
\pi(d(a)) = \partial_+(a)\omega_+ + \partial_-(a)\omega_-.
\end{gather*}

We construct a Dirac operator on $\Weyl p q$ by following the procedure of Beggs and Majid~\cite{BegMaj:spe}, which, in this particular case, is the generalization of a method employed in~\cite{Maj:Rie}. To start with, we identify the sections of a spinor bundle with the $\Weyl p q$-bimodule
$\AWeyl p q_{1} \oplus \AWeyl p q_{-1}$. More precisely, we set
\begin{gather}\label{spinor}
\mathcal{S}_+ = \AWeyl p q_{-1}{\mathsf{s}}_+, \qquad \mathcal{S}_- = \AWeyl p q_{1}{\mathsf{s}}_-, \qquad \mathcal{S} = \mathcal{S}_+\oplus \mathcal{S}_-,
\end{gather}
where ${\mathsf{s}}_\pm$ are formally understood as (central, i.e., commuting with elements of $\Weyl p q$) generators that distinguish components of the direct sum. In contrast to the cotangent bundle, the spinor bundle is trivial.

\begin{Lemma}\label{lemma.spinor}
The left $($resp.\ right$)$ $\Weyl p q$-module $\AWeyl p q_{1} \oplus \AWeyl p q_{-1}$ is free.
\end{Lemma}
\begin{proof}
Def\/ine
\begin{gather*}
\ell(1) := \frac{1}{p(0)} \left(\frac{p(0) - p\big(q^2 z\big)}{z} z_-\otimes z_+ + (-qx_-)\otimes \big({-}q^{-1}x_+\big)\right)\in \AWeyl p q_{-1}\otimes \AWeyl p q_{1}, \\
\ell(-1) := \frac{1}{p(0)} \left({z_+} \otimes \frac{p(0) - p(z)}{z} z_- + x_+\otimes x_-\right)\in \AWeyl p q_{1}\otimes \AWeyl p q_{-1} .
\end{gather*}
Relations \eqref{a.weyl} immediately imply that the multiplication applied to both $\ell(1)$ and $\ell(-1)$ gives~1, hence~$\ell(1)$ and $\ell(-1)$ are values of a strong connection in $\AWeyl p q$. Therefore, as explained in~\cite{BrzHaj:Che}, an idempotent corresponding to the projective module $\AWeyl p q_{1} $ can be computed as
\begin{gather*}
e(1) = \frac{1}{p(0)}\begin{pmatrix}
p(0) - p\big(q^2 z\big) & -qz_+x_- \vspace{1mm}\\
-q^{-1}x_+ \dfrac{p(0) - p\big(q^2 z\big)}{z} z_- & x_+x_-
\end{pmatrix}
=
\frac{1}{p(0)}\begin{pmatrix}
p(0) - p\big(q^2 z\big) & -x \vspace{1mm}\\
- \dfrac{p(0) - p(z)}{z} y &p(z)
\end{pmatrix},
\end{gather*}
while an idempotent of $\AWeyl p q_{-1} $, as
\begin{gather*}
e(-1) = \frac{1}{p(0)}\begin{pmatrix}
x_-x_+ & x_-z_+ \vspace{1mm}\\
 \dfrac{p(0) - p(z)}{z} z_-x_+ & p(0)-p(z)
\end{pmatrix}
=
\frac{1}{p(0)}\begin{pmatrix}
p\big(q^2 z\big) & x \vspace{1mm}\\
 \dfrac{p(0) - p(z)}{z} y & p(0)-p(z)
\end{pmatrix}.
\end{gather*}
Clearly,
\begin{gather*}
e(1)+e(-1) = \begin{pmatrix}
1 & 0\\ 0 &1 \end{pmatrix},
\end{gather*}
which proves the assertion.
\end{proof}

The strong connection forms $\ell(1)$, $\ell(-1)$ listed in the proof of Lemma~\ref{lemma.spinor} def\/ine a connection $\nabla\colon \mathcal{S} \to \bar\Omega \otimes \mathcal{S}$ on the spinor bundle $\mathcal{S}$ by the formula
\begin{gather*}
\nabla(a {\mathsf{s}}_+ + b {\mathsf{s}}_-) = \pi(d(a))\ell(-1) {\mathsf{s}}_+ + \pi(d(b))\ell(1) {\mathsf{s}}_-,
\end{gather*}
for all $a$ of degree $-1$ and $b$ of degree~1. Explicitly,
\begin{gather}
\nabla(a {\mathsf{s}}_+ + b {\mathsf{s}}_-) = \frac{q}{p(0)} \left(\partial_+(a){z_+}\omega_+ \otimes \frac{p(0) - p(z)}{z} z_-{\mathsf{s}}_+ + \partial_+(a)x_+\omega_+\otimes x_-{\mathsf{s}}_+ \right)\nonumber \\
\hphantom{\nabla(a {\mathsf{s}}_+ + b {\mathsf{s}}_-) =}{} + \frac{q}{p(0)} \left(\partial_-(a){z_+}\omega_- \otimes \frac{p(0) - p(z)}{z} z_-{\mathsf{s}}_+ + \partial_-(a)x_+\omega_-\otimes x_-{\mathsf{s}}_+ \right)\nonumber \\
 \hphantom{\nabla(a {\mathsf{s}}_+ + b {\mathsf{s}}_-) =}{}+ \frac{q^{-1}}{p(0)} \left(\partial_+(b){z_-}\frac{p(0) - p\left(q^2z\right)}{z} \omega_+ \otimes z_+{\mathsf{s}}_- + \partial_+(b)x_+\omega_+\otimes x_+{\mathsf{s}}_- \right)\nonumber \\
\hphantom{\nabla(a {\mathsf{s}}_+ + b {\mathsf{s}}_-) =}{} + \frac{q^{-1}}{p(0)} \left(\partial_-(b){z_-}\frac{p(0) - p\left(q^2z\right)}{z} \omega_- \otimes z_+{\mathsf{s}}_- + \partial_-(b)x_+\omega_-\otimes x_+{\mathsf{s}}_- \right) .\label{conn}
\end{gather}
The Clif\/ford action $\triangleright$ of $\bar\Omega$ on $\mathcal{S}$ is def\/ined, for all $a,b,c_\pm \in \AWeyl p q$ of degrees $|a| = -1$, $|b| =1$, $|c_\pm| = \pm 2$, by
\begin{gather}\label{Clifford}
(c_-\omega_+ + c_+\omega_-)\triangleright (a {\mathsf{s}}_+ + b {\mathsf{s}}_-) = \beta_+ c_-b {\mathsf{s}}_+ + \beta_- c_+a {\mathsf{s}}_-,
\end{gather}
where $\beta_+$, $\beta_-$ are (for the time being arbitrary) elements of the f\/ield ${\mathbb K}$. A connection together with a Clif\/ford action def\/ine the Dirac operator
\begin{gather*}
D := \triangleright \circ \nabla \colon \ \mathcal{S} \to \mathcal{S}.
\end{gather*}
Using relations \eqref{a.weyl}, the Dirac operator corresponding to data \eqref{conn}, \eqref{Clifford}, can be computed as
\begin{gather}\label{Dirac}
D(a {\mathsf{s}}_+ + b {\mathsf{s}}_-) = \beta_+ q^{-1}\partial_+(b) {\mathsf{s}}_+ + \beta_- q \partial_-(a) {\mathsf{s}}_- .
\end{gather}
$D$ is an even Dirac operator with the grading $\Weyl p q$-bimodule map def\/ined by
\begin{gather}\label{grading}
\gamma\colon \ \mathcal{S} \to \mathcal{S}, \qquad a {\mathsf{s}}_+ + b {\mathsf{s}}_- \longmapsto a {\mathsf{s}}_+ - b {\mathsf{s}}_-,
\end{gather}
since
\begin{align*}
D\circ \gamma (a {\mathsf{s}}_++b {\mathsf{s}}_-) &= D(a {\mathsf{s}}_+-b {\mathsf{s}}_-) = -\beta_+ q^{-1}\partial_+(b) {\mathsf{s}}_+ + \beta_- q \partial_-(a) {\mathsf{s}}_-\\
& = -\gamma(\beta_+ q^{-1}\partial_+(b) {\mathsf{s}}_+ + \beta_- q \partial_-(a) {\mathsf{s}}_-) = - \gamma\circ D(a {\mathsf{s}}_++b {\mathsf{s}}_-),
\end{align*}
i.e., $D$ anti-commutes with $\gamma$ as required.

\begin{Proposition}\label{prop.real}
Let ${\mathbb K} = {\mathbb C}$, $q\in (0,1)$ and $p$ be a $q^2$-separable polynomial with real coefficients. Equip the complex algebra $\AWeyl p q$ with the $*$-algebra structure as in~\eqref{star} and let $\Omega$ be the first-order calculus constructed in Theorem~{\rm \ref{thm.calc.A}} with the $*$-structure as in Remark~{\rm \ref{rem.star}}. View $\Weyl p q$ as a $*$-subalgebra of $\AWeyl p q$, let $\mathcal{S}$ be the module of sections of a spinor bundle \eqref{spinor}, and let $D$ be the Dirac operator~\eqref{Dirac} with the grading $\gamma$ given by~\eqref{grading}. Choose $\beta_\pm$ such that $\beta_-^*/\beta_+<0$, and let $\nu$ be a solution to
the equation
\begin{gather}\label{constraint}
\nu^2 = -q^3\frac{\beta_-^*}{\beta_+}.
\end{gather}
Then the linear map
\begin{gather*}
J \colon \ \mathcal{S} \to \mathcal{S}, \qquad a {\mathsf{s}}_++b {\mathsf{s}}_- \longmapsto -\nu^{-1} b^*{\mathsf{s}}_+ + \nu a^*{\mathsf{s}}_-,
\end{gather*}
equips $D$ with a real structure such that $D$ has ${\rm KO}$-dimension two.
\end{Proposition}

\begin{proof}
In order for $J$ to be a real structure for $D$ (of ${\rm KO}$-dimension two) it needs to satisfy the following f\/ive conditions:
\begin{subequations}\label{real}
\begin{gather}
J^2 = -\operatorname{id} \label{j2}
\\
J\circ\gamma = - \gamma\circ J, \label{jgamma}
\\
J\circ D = D\circ J, \label{jd}
\\
[u, J vJ] = 0, \qquad \mbox{for all}\ \quad u,v\in \Weyl p q, \label{order-0}
\\
[[D,u], J vJ] = 0, \qquad \mbox{for all} \quad u,v\in \Weyl p q, \label{order-1}
\end{gather}
\end{subequations}
where $[-,-]$ denotes the commutator and elements of $\Weyl p q$ act on $\mathcal{S}$ by left multiplication; cf.~\cite{Con:rea}. The f\/irst two conditions are easy to check:
\begin{gather*}
J\circ J(a{\mathsf{s}}_++b{\mathsf{s}}_-) = J\big({-}\nu^{-1} b^*{\mathsf{s}}_+ + \nu a^*{\mathsf{s}}_-\big) = - a {\mathsf{s}}_+-b{\mathsf{s}}_-,
\end{gather*}
by the involution property of $*$, and
\begin{gather*}
J\circ \gamma (a{\mathsf{s}}_++b{\mathsf{s}}_-) = J(a{\mathsf{s}}_+-b{\mathsf{s}}_-) =\nu^{-1} b^*{\mathsf{s}}_+ + \nu a^*{\mathsf{s}}_-,
\\
\gamma\circ J (a{\mathsf{s}}_++b{\mathsf{s}}_-) = \gamma\big({-}\nu^{-1} b^*{\mathsf{s}}_+ + \nu a^*{\mathsf{s}}_-\big) = -\nu^{-1} b^*{\mathsf{s}}_+ - \nu a^*{\mathsf{s}}_- = -J\circ \gamma (a {\mathsf{s}}_++b{\mathsf{s}}_-) ,
\end{gather*}
as required.

To check \eqref{jd} we use the def\/initions of $D$ and $J$, and compute, for all $a\in \AWeyl p q_{-1}$, $b\in \AWeyl p q_{1}$,
\begin{gather*}
D\circ J(a {\mathsf{s}}_++b {\mathsf{s}}_-) = D(-\nu^{-1} b^*{\mathsf{s}}_+ + \nu a^*{\mathsf{s}}_-)
 = q^{-1}\beta_+\nu\partial_+ (a^* ) {\mathsf{s}}_+ - q\beta_-\nu^{-1}\partial_- (b^* ) {\mathsf{s}}_-,
\end{gather*}
and
\begin{align}
J\circ D(a {\mathsf{s}}_++b {\mathsf{s}}_-) &= J(\beta_+ q^{-1}\partial_+(b) {\mathsf{s}}_+ + \beta_- q \partial_-(a) {\mathsf{s}}_-)\nonumber \\
& = -q^{-1}\beta^*_-\nu^{-1}\partial_-(a)^*{\mathsf{s}}_+ + q\beta^*_+\nu\partial_+(b)^* {\mathsf{s}}_-.\label{jd.calc}
\end{align}
Note that, in view of the def\/inition \eqref{star}, if $c\in \AWeyl p q$ is a homogeneous element, then $|c^*| = -|c|$. Applying the involution $*$ to the def\/inition of $d$, using the $*$-structure on $\Omega$ in Remark~\ref{rem.star} and then the def\/initions of $\sigma_\pm$~\eqref{sigma} one f\/inds that, for all homogeneous~$c\in \AWeyl p q$,
\begin{gather*}
\partial_+(c)^* = q^{|c|-2}\partial_- (c^* ), \qquad \partial_-(c)^* = q^{|c|+2}\partial_+ (c^* ).
\end{gather*}
Since $|a|=-1$ and $|b|=1$, equality \eqref{jd.calc} can be developed further to give
\begin{gather*}
J\circ D(a {\mathsf{s}}_++b {\mathsf{s}}_-) = -q^2\nu^{-1}\beta_-^*\partial_+(a^*) {\mathsf{s}}_+ + q^{-2}\nu\beta_+^*\partial_-(b^*) {\mathsf{s}}_-,
\end{gather*}
and so $J\circ D = D\circ J$ by \eqref{constraint}.

To check the order-zero condition \eqref{order-0}, take any $a\in \AWeyl p q_{-1}$ and $u,v\in \Weyl p q$, and compute
\begin{align*}
[u, J vJ](a {\mathsf{s}}_+) &= uJ ( vJ (a {\mathsf{s}}_+ ) ) - J ( vJ (ua {\mathsf{s}}_+ ) )
 = \nu (u J (v a^*{\mathsf{s}}_- ) - J (v a^*u^*{\mathsf{s}}_- ) )\\
&= -uav^*{\mathsf{s}}_+ +uav^*{\mathsf{s}}_+ =0.
\end{align*}
In a similar way one proves that also for all $b\in \AWeyl p q_{1}$ and $u,v\in \Weyl p q$,
\begin{gather*}
[u, J vJ](b {\mathsf{s}}_-)=0.
\end{gather*}
 Finally, as the f\/irst step towards proving the order-one condition \eqref{order-1}, observe that the skew derivation property of $\partial_\pm$ implies that, for all $u\in \Weyl p q$, $a\in \AWeyl p q_{-1}$ and $b\in \AWeyl p q_{1}$,
\begin{gather*}
[D,u](a {\mathsf{s}}_++b {\mathsf{s}}_-) = \beta_+\partial_+(u)b {\mathsf{s}}_++\beta_-\partial_-(u)a {\mathsf{s}}_-,
\end{gather*}
hence, for all $v\in \Weyl p q$,
\begin{align*}
[[D,u], J vJ] (a {\mathsf{s}}_+) &= - [D,u](av^*{\mathsf{s}}_+) - \beta_- J(vJ(\partial_-(u)a {\mathsf{s}}_-))\\
&= -\beta_-\partial_-(u)av^*{\mathsf{s}}_- +\beta_-\partial_-(u)av^*{\mathsf{s}}_- =0,
\end{align*}
and similarly $[[D,u], J vJ] (b {\mathsf{s}}_-) =0$. The distribution of signs on the right hand sides of relations \eqref{j2}, \eqref{jgamma} and \eqref{jd} indicates that $D$ is a Dirac operator with a real structure of ${\rm KO}$-dimension two, as stated.
\end{proof}

\begin{Remark}
There is some level of arbitrariness in the f\/ixing of the ${\rm KO}$-dimension of $D$. Choosing the parameters $\beta_\pm$ in such a way that $\beta_-^*/\beta_+>0$, will result in the change of the sign on the right hand side of equation \eqref{j2}. This distribution of signs on the right hand sides of relations \eqref{j2}, \eqref{jgamma} and \eqref{jd} corresponds to the ${\rm KO}$-dimension being~6 (modulo~8).
\end{Remark}

\section{Outlook}

The paper the reader is presented with herein is concerned with f\/irst-order dif\/ferential (and integral) calculi on a class of generalized Weyl algebras. A detailed study of higher forms on~$\AWeyl p q$ and~$\Weyl p q$, both dif\/ferential and integral, is a natural next step. Every f\/irst-order dif\/ferential calculus admits an extension to a full dif\/ferential graded algebra. Such a universal extension might be trivial, might lead to the dif\/ferential structure of classical dimensions or it might be very large, bearing no resemblance to what can be expected from the classical case. In the latter case, further quotients that reduce the dimension of modules of higher forms might be possible. In the cases studied in the present paper, it would be interesting and indeed desired to f\/ind out whether the calculus on a regular generalized Weyl algebra~$\AWeyl p q$ described in Theorem~\ref{thm.calc.A} admits a dif\/ferential graded algebra of `classical dimensions', i.e., such that the module of two-forms is free of rank~3 and the module of three-forms is free of rank~1, with no higher forms than three-forms (in other words: the calculus admits a volume three-form). Should such a full dif\/ferential calculus exist, does the divergence constructed in Proposition~\ref{prop.int} extend to a f\/lat hom-connection so as to produce a complex of integral forms? Finally, is this complex isomorphic to the de Rham complex of dif\/ferential forms, i.e., is the dif\/ferential calculus integrable? Since~$\AWeyl p q$ is an af\/f\/ine algebra of Gelfand--Kirillov dimension three, the existence of such an isomorphism would establish {\em differential smoothness} of~$\AWeyl p q$ in the sense of~\cite{BrzSit:smo}.

Similar questions can and should be asked about the f\/irst-order dif\/ferential calculus $(\bar{\Omega}, d)$ on~$\Weyl p q$. In this case, the extension should contain only two-forms, and the module of two-forms should be free of rank one. This is dictated by the Gelfand--Kirillov dimension of~$\Weyl p q$. It is natural to enquire in this case whether the separation of one-forms into holomorphic and anti-holomorphic components carries on to the higher calculus so as to deliver a complex structure on~$\Weyl p q$~\cite{BegSmi:com,KhaLan:hol}. Again the question about dif\/ferential smoothness of~$\Weyl p q$ should be addressed.

It might be that in addition to the $q^2$-separability another condition on $p$ should be imposed to guarantee the integrability of dif\/ferential calculi on $\AWeyl p q$ and $\Weyl p q$, and to establish dif\/ferential smoothness of these algebras. As explained in \cite{Liu:hom}, generalized Weyl algebras $\Weyl p q$ are homologically smooth (i.e., they have f\/inite length resolutions by f\/initely generated projective bimodules) provided the def\/ining polynomial $p$ has no repeated roots (in case of an algebraically closed f\/ield this is equivalent to the separability of ~$p$). Should no additional (to $q^2$-separability) conditions on $p$ be required or should required additional conditions be dif\/ferent from separability, the regular generalized Weyl algebras $\Weyl p q$ could serve as examples of algebras that are dif\/ferentially but not necessarily homologically smooth. A relationship between these two forms of (noncommutative) smoothness is not yet understood.

The construction of a real Dirac operator on $\Weyl p q$ presented above is purely algebraic. An obvious next step is to study analytic aspects of this construction in order to form real spectral triples on~$\Weyl p q$ (as opposed to algebraic Dirac operators with a real structure hitherto described). In the f\/irst instance one should study the theory of $*$-representations of~$\Weyl p q$, the rudiments of which are outlined in~\cite{Brz:gen} in order to construct a suitable Hilbert space. This process might force one to impose additional conditions on the polynomial~$p$. One can also follow the procedure of Beggs and Majid employed successfully in~\cite{BegMaj:spe} in the case of the quantum sphere and the quantum disc. This procedure makes use of the integral on the studied algebra. The integral~$\Lambda$ on~$\AWeyl p q$ constructed in Proposition~\ref{prop.int} can be restricted to~$\Weyl p q$. Since the integral space is f\/inite but not one-dimensional, before the restriction of $\Lambda$ can be used to obtain an inner product, a suitable hermitian inner product on the integral space need be constructed.

These are topics for future work.

\subsection*{Acknowledgements}
The author would like to express his gratitude to the referees for many helpful and detailed comments and suggestions.

\pdfbookmark[1]{References}{ref}
\LastPageEnding

\end{document}